\documentclass[reqno,b5paper]{amsart}
\usepackage{amsmath}
\usepackage{amssymb}
\usepackage{amsthm}
\usepackage{enumerate}
\usepackage[mathscr]{eucal}
\theoremstyle{plain}
\newtheorem{thm}{Theorem}[section]

\newtheorem{Example}{Example}[section]

\theoremstyle{definition}
\newtheorem{defn}{Definition}[section]

\begin{document}

\setcounter {page}{1}
\title{a note on rough statistical convergence}

\author[M. Maity]{Manojit Maity\ }
\newcommand{\acr}{\newline\indent}
\maketitle
\address{25 Teachers Housing Estate, P.O.- Panchasayar, Kolkata-700094, West Bengal, India. Email: mepsilon@gmail.com\\}

\maketitle
\begin{abstract}
In this paper, we introduce the notions of pointwise rough statistical convergence and rough statistically Cauchy sequences of real valued functions in the line of A. (T$\ddot{u}$rkmenoglu) G$\ddot{o}$khan and M. G$\ddot{u}$ng$\ddot{o}$r \cite{Tu1}. Furthermore we study thier equivalence.
\end{abstract}
\author{}
\maketitle



\section{\textbf{Introduction:}} The idea of statistical convergence was first introduced by Fast \cite{Fa} and also by 
Schoenberg\cite{Sc} independently, using the idea of natural density. Let $K$ be a subset of the set of positive integers $\mathbb{N}$. Let $K_n$ be a set defined as follows, $K_n = \{k \in K : k \leq n\}$. Then the natural density of $K$ is defined as $d(K) = \underset{n \rightarrow \infty}{\lim} \frac{|K_n|}{n}$, where $|K_n|$ denotes the number of elements in $K_n$. Clearly finite set has natural density zero.\\
A real sequence $\{x_n\}_{n \in \mathbb{N}}$ is said to be pointwise statistically convergent ti $\xi$ on a set $S$ if for every $\varepsilon >0$, $\underset{n \rightarrow \infty}{\lim}|\{k \leq n : |x_n - \xi| \geq \varepsilon~ \mbox{for }~\xi \in S\}| = 0$. \\
The idea of rough convergence was introduced by Phu \cite{Phu1} in finite dimensional normed linear spaces. If $x = \{x_n\}_{n \in \mathbb{N}}$ is a sequence of real numbers and $r$ is a nonnegative real number, then $x$ is said to be rough convergent to $\xi \in \mathbb{R}$ if for every $\varepsilon > 0 $ there exists $N \in \mathbb{N}$ such that 
\begin{center}
$|x_n - \xi| < r + \varepsilon$ for all  $n \geq N$.
\end{center}

A lot of work has been done in this area by Phu \cite{Phu1},\cite{Phu2}, Aytar \cite{Ay}, Pal, Chandra and Dutta \cite{Pa} , Malik and Maity \cite{Ma1}, \cite{Ma2}.\\

Recently the idea of pointwise statistical convergence of a sequence of real valued functions is introduced by G$\ddot{o}$khan and G$\ddot{u}$ng$\ddot{o}$r \cite{Tu1}. Some works in this line can be found in \cite{Ci}, \cite{Tu2}.
So it is quite natural to think about whether the notion of rough convergence can be introduced for the pointwise statistical convergence of real valued functions. In this paper we do the same and introduce the notion of pointwise rough statistical convergence of sequence of real valued functions and study some basic properties of this type of convergence.

\section{\textbf{Rough Statistical Convergent Sequence of Real Valued Functions}}
\begin{defn}
A sequence of real valued functions $\{f_k\}_{k \in \mathbb{N}}$ on a set $ X $ is said to be pointwise rough statistically convergent to a function $f$ on a set $A$ if for every $ \varepsilon > 0 $ there exists a real number $ r > 0 $ such that
\begin{eqnarray*}
\underset{ n \rightarrow \infty}{\lim} \frac{1}{n} | \{ k \leq n : |{f_k}(x) - f(x)| \geq r + \varepsilon \}| = 0 
\end{eqnarray*}
for every $ x \in X $.\\
That is 
\begin{eqnarray}
|{f_k}(x) - f(x)| < r + \varepsilon ~~ a.a.k.
\end{eqnarray}
for every $ x \in X $
\end{defn}

In this case we write $ r\mbox{-}st\mbox{-}\lim {f_k}(x) = f(x) $ on $ A $. This means that for every $ \delta > 0 $ there exists an integer $ N $ such that $ \frac{1}{n}| \{ k \leq n : |{f_k}(x) - f(x)| \geq r + \varepsilon \}| < \delta $ for every $ x \in A $ and for all $ n > N (= N(\varepsilon,\delta,x)) $ and for every $ \varepsilon > 0 $.

Now we define a sequence of real valued functions which is not pointwise statistically convergent but pointwise rough statistically convergent.

\begin{Example}
Define a sequence of real valued functions $\{f_i\}_{i \in \mathbb{N}}$ by
\begin{eqnarray*}
f_i(x) &=& \frac{1}{1+x^i} ~~~\mbox{if}~~~ i \neq k^2, k = 1,2,3,\cdots\\
       &=& i, ~~~\mbox{otherwise}
\end{eqnarray*}
on $[0,1]$.
\end{Example}
The sequence $ \{f_i\}_{i \in \mathbb{N}} $ is pointwise rough statistically convergent to zero with roughness degree 1.
But this sequence is not pointwise statistically convergent.
\begin{thm}
Let $\{f_k\}_{k \in \mathbb{N}}$ and $\{g_k\}_{k \in \mathbb{N}}$ be two sequences of real valued functions defined on a set $X$. If $r\mbox{-}st\mbox{-}\lim {f_k}(x) = f(x)$ and $r\mbox{-}st\mbox{-}\lim {g_k}(x) = g(x)$ on $X$, then $r\mbox{-}st\mbox{-}\lim({\alpha}{f_k}(x) + {\beta}{g_k}(x)) = {\alpha}{f_k}(x) + {\beta}{g_k}(x) $ for all $\alpha, \beta \in \mathbb{R}$. 
\end{thm}

\begin{proof}
The proof is immediate for $\alpha = 0$ and $\beta = 0$. Let $\alpha \neq 0$ and $\beta \neq 0$. Let $r>0$ be a real number. Let $\varepsilon > 0$ be given. Then \\
$$\{ k \leq n: |{\alpha}{f_k}(x) + {\beta}{g_k}(x) - ({\alpha}f(x) + {\beta}g(x))| \geq r + \varepsilon, ~~\mbox{for any}~~ x \in X\}$$
$$\subseteq \{ k \leq n: |{\alpha}{f_k}(x) - {\alpha}f(x)| \geq r + \varepsilon, ~~\mbox{for any}~~ x \in X\} \bigcup $$ 
$$\{ k \leq n: |{\beta}{g_k}(x) - {\beta}g(x))| \geq r + \varepsilon, ~~\mbox{for any}~~ x \in X\}.$$

Since $|{\alpha}{f_k}(x) + {\beta}{g_k}(x) - {\alpha}f(x) + {\beta}g(x))| \geq\ |{\alpha}{f_k}(x) - ({\alpha}f(x)| +  |{\beta}{g_k}(x) - {\beta}g(x)|$, hence we obtain $r\mbox{-}st\mbox{-}\lim ({\alpha}{f_k}(x) + {\beta}{g_k}(x)) = {\alpha}f(x) + {\beta}g(x) ~~ \mbox{on}~~ X $.
\end{proof}

Now we introduce the rough statistical analog of the Cauchy convergence criterion.


\begin{defn}
Let $\{f_k\}_{k \in \mathbb{N}}$ be a sequence of real valued functions defined on a set $X$. Let $r>0$ be a real number. The sequence $\{f_k\}_{k \in \mathbb{N}}$ is said to be rough statistically Cauchy sequence if for every $\varepsilon > 0$ there exists a natural number $N(=N(\varepsilon, x))$ such that $|{f_k}(x) - {f_N}(x)| < r + \varepsilon $ a.a.k. that is $ \underset{n \rightarrow \infty}{\lim}\frac{1}{n} |\{k \leq n : |{f_k}(x) - {f_N}(x)| \geq r + \varepsilon| = 0$ for every $ x \in X$. 
\end{defn}

\begin{thm}
Let $\{f_k\}_{k \in \mathbb{N}}$ be a sequence of real valued function defined on a set $X$. Then the following statements are equivalent.\\
(i) $\{f_k\}_{k \in \mathbb{N}}$ is a pointwise rough statistically convergent sequence on $X$.\\
(ii) $\{f_k\}_{k \in \mathbb{N}}$ is a rough statistically Cauchy sequence on $X$.\\
(iii) $\{f_k\}_{k \in \mathbb{N}}$ is a sequence of real valued functions on $X$ for which there exists a point wise rough statistically convergent sequence of real valued functions $\{g_k\}_{k \in \mathbb{N}}$ such that ${f_k}(x) = {g_k}(x)$ a.a.k. for every $x \in X$.
\end{thm}

\begin{proof}
(i) implies (ii)\\

Let $r > 0$ be a real number. Let $ r\mbox{-}st\mbox{-}\lim{f_k}(x) = f(x)$. Then for every $\varepsilon > 0$ $ |{f_k}(x) - f(x)| < \frac{r + \varepsilon}{2}$ for every $x \in X$ a.a.k. Let $N$ be a natural number so chosen that $|{f_N}(x) - f(x)| < \frac{r + \varepsilon}{2}$ for every $x \in X$. Then we have $|{f_k}(x) - {f_N}(x)| \leq |{f_k}(x) - f(x)| + |{f_N}(x) - f(x)| < {\frac{r + \varepsilon}{2}} + {\frac{r + \varepsilon}{2}} = {r + \varepsilon} $ a.a.k. for every $x \in X$. Hence $\{f_k\}_{k \in \mathbb{N}}$ is a rough statistically Cauchy sequence.\\

(ii) implies (iii)\\

Assume that (ii) is true. Choose a natural number $N$ such that the band $[{f_N}(x)-1,{f_N}(x)+1] = I$ contains ${f_k}(x)$ a.a.k. for every $x \in X$. By (ii) we get a natural number $M$ such that $[{f_M}(x)- \frac{1}{2}, {f_M}(x)+  \frac{1}{2}] = I^{'}$ contains ${f_k}(x)$ a.a.k. for every $x \in X$. Hence $ I \cap I^{'} = I_1 $ contains ${f_k}(x)$ a.a.k. for every $ x \in X $. Now $\{k \leq n : {f_k}(x) \notin I \cap I^{'} , ~~ \mbox{for every}~~ x \in X \} = \{k \leq n : {f_k}(x) \notin I ,~~ \mbox{for every} ~~ x \in X \} \bigcup \{k \leq n : {f_k}(x) \notin I^{'} ,~~ \mbox{ for every} ~~ x \in X \}$.\\
Hence $ \underset{n \rightarrow \infty}{\lim} \frac{1}{n}|\{k \leq n: {f_k}(x) \notin I \cap I^{'} ~~ \mbox{for every} x \in X \}| \leq \underset{n \rightarrow \infty}{\lim} \frac{1}{n}|\{k \leq n: {f_k}(x) \notin I ~~ \mbox{for every} x \in X \}| + \underset{n \rightarrow \infty}{\lim} \frac{1}{n}|\{k \leq n: {f_k}(x) \notin I^{'} ~~ \mbox{for every} x \in X \}| = 0 $.\\
Therefore $ I_1 $ is a closed band of height less than or equal to 1 which contains $ {f_k}(x) $ a.a.k. for every $ x \in X $.\\
Proceeding in the same manner we can get $ N_2 $ such that $ I^{''} = [{f_{N_2}}(x) - \frac{1}{4} , {f_{N_2}}(x) + \frac{1}{4}] $ contains $ {f_k}(x) $ a.a.k. and by above argument $ I \cap I^{''} = I_2 $ contains $ {f_k}(x) $ a.a.k for every $ x \in X $ and the height of $ I_2 $ is less than or equal to $ \frac{1}{2} $. By induction principle we can costruct a sequence $ \{I_m\}_{m \in \mathbb{N}} $ of closed band such that for each $ m $, $ I_m \supseteq I_{m+1} $ and the height of $ I_m $ is not greater than $ 2^{1-m} $ and $ {f_k}(x) \in I_m $ a.a.k. for every $ x \in X $. Thus there exists a function $ f(x) $ defined on $ X $ such that $ \{ f(x) \} $ is equal to $ \bigcap \limits_{m=1}^{\infty}{I_m} $. Using the fact that $ {f_k}(x) \in I_m $ a.a.k. for every $ x \in X $, we can choose an increasing positive integer sequence $\{J_m\}_{m \in \mathbb{N}}$ such that
\begin{eqnarray}
\frac{1}{n}|\{k \leq n:{f_k}(x) \notin I_m ~~\mbox{ for every}~~ x \in X \}| < \frac{1}{m} ~~\mbox{if}~~ n > J_m
\end{eqnarray}
Now construct a subsequence $ \{{h_k}(x)\} $ of $ \{{f_k}(x)\} $ containing of all terms $ {f_k}(x) $ such that $ k > J_1 $ and when $ J_m < k \leq J_{m+1} $ then $ {f_k}(x) \notin I_m $ for every $ x \in X $.

Define a sequence of real valued function $ \{ {g_k}(x) \} $ by 
\begin{eqnarray*}
{g_k}(x) &=& f(x), ~~ \mbox{if}~~ {f_k}(x) ~~\mbox{ is a term of }~~ \{{h_k}(x)\} \\
         &=& {f_k}(x), ~~ \mbox{otherwise}
\end{eqnarray*}
for every $ x \in X $. Then $ \underset{ k \rightarrow \infty }{\lim} {g_k}(x) = f(x) $ on $X$. If $ \varepsilon > \frac{1}{m} > 0 $ and $ k > J_m $, then either $ {f_k}(x) $ is a term of $\{{h_k}(x)\}$, which means
\begin{eqnarray*}
            {g_k}(x) &=& f(x) ~~ \mbox{on}~~ X \\
\mbox{or}~~ {g_k}(x) &=& {f_k}(x) \in I_m  ~~\mbox{on}~~ X.
\end{eqnarray*} 
and $|{g_k}(x) - f(x)| \leq ~~\mbox{height of }~~ I_m \leq 2^{1-m}$ for every $x \in X$. 

Now if $ J_m < n < J_{m+1} $, then 
\begin{eqnarray*}
\{k \leq n:{g_k}(x) \neq {f_k}(x) ~~\mbox{for every}~~ x \in X \} \subseteq \{k \leq n:{f_k}(x) \notin I_m ~~\mbox{for every}~~ x \in X \}
\end{eqnarray*}
so by (10)\\
\begin{eqnarray*}
\frac{1}{n} |\{k \leq n:{g_k}(x) \neq {f_k}(x) ~~\mbox{for every}~~ x \in X \}|
&\leq& \frac{1}{n} |\{k \leq n:{f_k}(x) \notin I_m ~~\mbox{for every}~~ x \in X \}|\\ 
&<&\frac{1}{m}.
\end{eqnarray*}
By taking limit $ n \rightarrow \infty $ we get the limit is 0 and consequently we say that ${f_k}(x) = {g_k}(x)$ a.a.k. for every $x \in X$. Therefore (ii) implies (iii).\\

(iii) implies (i)\\

Let us assume that (iii) holds. That is $ {f_k}(x) = {g_k}(x) $ a.a.k. for every $ x \in X $ and $\underset{ k \rightarrow \infty}{\lim} {g_k}(x) = f(x)$ on $X$. Let $ r > 0 $ be real number and $\varepsilon > 0$. Then for every $n$, $\{k \leq n : |{f_k}(x) - f(x)|\geq r + \varepsilon ~~\mbox{ for every }~~ x \in X \} \subseteq \{k \leq n: {f_k}(x) \neq {g_k}(x) ~~\mbox{ for every } ~~ x \in X \} \bigcup \{k \leq n : |{g_k}(x) - f(x)| \geq r + \varepsilon ~~\mbox{ for every }~~ x \in X \}$; 
since $ \underset{k \rightarrow \infty}{\lim}{g_k}(x) = f(x) $ on $X$, the later set contains a finite number of integers, say $ p = p(\varepsilon,x) $.\\
Therefore since ${f_k}(x) = {g_k}(x)$ a.a.k for every $ x \in X $ we get
\begin{center}
$\underset{ n \rightarrow \infty}{\lim} \frac{1}{n}|\{k \leq n:|{f_k}(x) - f(x)| \geq r + \varepsilon ~~ \mbox{ for every } ~~ x \in X \}|$ \\
$\leq \underset{n \rightarrow \infty}{\lim} \frac{1}{n}|\{k \leq n:{f_k}(x) \neq {g_k}(x)| ~~ \mbox{ for every } ~~ x \in X \}| + \underset{ n \rightarrow \infty}{\lim} \frac{p}{n}= 0$
\end{center}
Hence $ |{f_k}(x)-f(x)| < r + \varepsilon $ a.a.k for every $ x \in X $, so (i) holds.
\end{proof}

Acknowledgement: I express my gratitude to Prof. Salih Aytar, Suleyman Demirel University, Turkey, for his paper entitled ``Rough Statistical Convergence'' which inspired me to prepare and develope this paper. I also express my gratitude to Prof. Pratulananda Das, Jadavpur University, India and Prof. Prasanta Malik, Burdwan University, India for their advice in preparation of this paper.

\end{document}